\documentclass{amsart}
\usepackage{amsmath}
\usepackage{amsfonts}
\usepackage{amssymb}
\usepackage[utf8]{inputenc}
\usepackage{amssymb}
\usepackage{mathrsfs}
\usepackage{amsthm}
\usepackage{xcolor}

\usepackage{enumerate}

\usepackage{graphicx}

\newtheorem{thrm}{Theorem}[section]

\newtheorem{lemma}[thrm]{Lemma}
\newtheorem{prop}[thrm]{Proposition}

\newtheorem{question}[thrm]{Question}
\theoremstyle{definition}
\newtheorem{defin}[thrm]{Definition}

\newtheorem{claim}{Claim}

\DeclareMathOperator{\ex}{exp}

\newcommand{\FF}{\mathscr{F}}
\newcommand{\EE}{\mathscr{E}}
\newcommand{\AAA}{\mathscr{A}}
\newcommand{\BB}{\mathscr{B}}
\newcommand{\UU}{\mathscr{U}}
\newcommand{\VV}{\mathscr{V}}

\newcommand{\WW}{\mathscr{W}}

\newcommand{\KK}{\mathcal{K}}

\begin{document}

\title{Lindel\"of scattered $W$-spaces are $\sigma$-compact}

\author[M. Krupski]{Miko\l aj Krupski}
%\date{\today}
\address{
Institute of Mathematics\\ University of Warsaw\\ ul. Banacha 2\\
02--097 Warszawa, Poland }
\email{mkrupski@mimuw.edu.pl}

\begin{abstract}
    We show that every Lindel\"of scattered $W$-space is $\sigma$-compact. This result generalizes a theorem proved recently by Avil\'es and the author in [Topology Appl. 363 (2025), Paper No. 109234] and answers a question posed by Tkachuk. 
\end{abstract}

\maketitle

\section{Introduction}

All spaces under consideration are assumed to be regular.
In this paper, we study the problem, considered in \cite{Tk1} and subsequently in \cite{AK}, of when a Lindel\"of scattered space $X$ is $\sigma$-compact. The main theorem of \cite{AK}, which improves an earlier result established by Tkachuk in \cite{Tk1}, reads as follows:
\begin{thrm}\cite[Theorem 1.2]{AK}\label{thm Aviles, K.}
If $X$ is a Lindel\"of scattered subspace of a $\Sigma$-product $\Sigma(\prod_{\gamma\in \Gamma}Y_\gamma, a)$, of first-countable spaces $Y_\gamma$, then $X$ is $\sigma$-compact.
\end{thrm}

The assumptions on $X$ in the above theorem involves an external structure, namely, the $\Sigma$-product. One may wonder if Theorem \ref{thm Aviles, K.} can be generalized in such a way that $\sigma$-compactness of a Lindel\"of scattered space $X$ follows from an internal topological property of $X$.

The notion of a $W$-space was introduced by Gruenhage in \cite{Gr1} as a generalization of first-countability. It is defined in terms of a two-player topological game, called the $W$-game. Let us recall that given a space $X$ and a point $x\in X$, the \textit{$W$-game at $x$} is a game with $\omega$-many innings, played by two players: player I and player II. In the $n$-th round of the game ($n=1,2,\ldots$), player I picks on open neighborhood $U_n$ of $x$ in $X$ and player II responds by choosing a point $x_n\in U_n$. Player I wins the game if the sequence $(x_n)_{n=1}^\infty$ converges to $x$. Otherwise, player II wins. We say that a point $x\in X$ is a \textit{$W$-point} if player I has a winning strategy in the $W$-game at $x$. A space $X$ is a \textit{$W$-space} if all of its points are $W$-points.

It is easy to see that any first countable space is a $W$-space and any subspace of a $W$-space is a $W$-space itself. It is also known that any $\Sigma$-product of a family of $W$-spaces is a $W$-space (see \cite{Gr}). This motivated Tkachuk to ask the following natural question (see \cite{Tk2}): 

\begin{question}(Tkachuk)\label{question}
Let $X$ be a Lindel\"of scattered $W$-space. Is $X$ $\sigma$-compact? 
\end{question}

The purpose of this note is to answer the above question in the affirmative. 
We will prove the following:

\begin{thrm}\label{main}
If $X$ is a Lindel\"of scattered $W$-space, then $X$ is $\sigma$-compact.
\end{thrm}

\section{Notation}

We use standard set-theoretic and topological notation.
For a set $X$, by $[X]^{<\omega}$ we denote the family of all finite subsets of $X$, whereas the symbol $X^{<\omega}$ stands for all finite sequences in $X$. By $\ex(X)$ we will denote the power set of $X$, i.e., the family of all subsets of $X$.
If $X$ is a topological space and if $x\in X$, then by $\tau(X,x)$ we denote the family of all open neighborhoods of $x$ in $X$. 

We follow Gruenhage's terminology \cite{Gr} for the $W$-game. Since the strategy $\sigma$ of the $W$-game at $x$ for player I depends only on the moves of player II, we can view such a strategy as a function $\sigma:X^{<\omega}\to \tau(X,x)$ that assigns to each finite sequence in $X$ an open neighborhood of $x$. We say that a finite sequence $s=(x_1,\ldots,x_n)\in X^{<\omega}$ is a \textit{$\sigma$-sequence} if $s=\emptyset$ or $x_{1}\in \sigma(\emptyset)$ and $x_{i+1}\in \sigma(x_1,\ldots ,x_i)$ for $i=1,\ldots ,n-1$. An infinite sequence $(x_1,x_2,\ldots)$ is a \textit{$\sigma$-sequence} if all of its initial segments are $\sigma$-sequences. The strategy $\sigma$ is winning if every infinite $\sigma$-sequence converges to $x$. Note that if $s\in X^{<\omega}$ is not a $\sigma$-sequence, then we may define $\sigma(s)$ arbitrarily, say $\sigma(s)=X$.

In the sequel, the $W$-game will be used in the following context. We will work with a pair of spaces $X\subseteq K$ such that $X$ is dense in $K$ (in fact, $K$ will be a compactification of $X$) and $p$ will be a $W$-point in the space $X$. If $\sigma$ is a winning strategy for player I in the $W$-game at $p$, then we can assume that for any $s\in X^{<\omega}$, the set $\sigma(s)$ is a neighborhood of $p$ in $K$, i.e., $\sigma(s)\in \tau(K,p)$. Indeed, one can consider a new game in which, in the $n$-th round, player I picks a set $V_n\in \tau(K,p)$ and player II responds by choosing a point $x_n\in X\cap V_n$ (this can be done because $X$ is dense in $K$). It is readily seen that player I has a winning strategy in this new game if and only if player I has a winning strategy in the $W$-game at $p\in X$. Thus, in the situation described above we will always assume that the moves of player I, given by a strategy in the $W$-game, are open in $K$.

For a given space $X$, by $X'$ we denote the set of all accumulation points of $X$. For an ordinal $\alpha$ we define the \textit{Cantor-Bendixson $\alpha$-th derivative} $X^{(\alpha)}$ recursively, as follows:
\begin{itemize}
\item $X^{(0)}=X$
\item $X^{(\alpha+1)}=(X^{(\alpha)})'$
\item $X^{(\alpha)}=\bigcap_{\beta<\alpha}X^{(\beta)}$ if $\alpha$ is a limit ordinal.
\end{itemize}

The space $X$ is scattered if and only if $X^{(\alpha)}=\emptyset$ for some ordinal $\alpha$. For any scattered space $X$, we define its \textit{Cantor-Bendixson rank $ht(X)$} as the minimal ordinal $\alpha$ for which $X^{(\alpha)}=\emptyset$. It is worth noting that if $X$ is scattered compact, then $ht(X)$ is always a successor cardinal. 

A set that is both closed and open is referred to as a \textit{clopen set}. A space $X$ is \textit{zero-dimensional} if it has a base consisting of clopen sets.

\section{Results}

Our approach to prove Theorem \ref{main} is similar to the one presented in \cite{AK}. The main difference is that we replace \cite[Lemma 2.4]{AK} with Lemma \ref{main lemma} below, whose proof adapts some ideas of Gruenhage from \cite{Gr}. The following lemma, also used in \cite{AK}, is due to Tkachuk \cite{Tk1}.
\begin{lemma}\cite[Lemma 3.1]{Tk1}\label{lemma Tkachuk}
Suppose that $K$ is a space and let $X\subseteq K$ be a subspace of $K$. Let $\{G_t:t\in T\}$ be an indexed family of $G_\delta$-subsets of $K$ such that $G_t\cap X=\emptyset$ for all $t\in T$. Assume further that there is a family $\{U_t:t\in T\}$ of open subsets of $K$ such that $G_t\subseteq U_t$ for all $t\in T$ and, for every $x\in X$, the set $\{t\in T:x\in U_t\}$ is finite. Then there exists a $G_\delta$-subset $G$ of $K$ such that $\bigcup_{t\in T}G_t\subseteq G\subseteq K\setminus X$.
\end{lemma}

\begin{defin}
    Let $X$ be a subspace of a space $K$ and let $\UU$ be a family of subsets of $K$. We say that $\UU$ is \textit{point-finite in $X$} if for every $x\in X$ the collection $\{U\in \UU:x\in U\}$ is finite.
\end{defin}

\begin{defin}
Let $M$ be a set and let $\phi:\ex(M)\to \ex(M)$ be a function. A subset $A\subseteq M$ is called \textit{$\phi$-closed} if $\phi(A)\subseteq A$.
\end{defin}

The following fact was noted by Gruenhage.
\begin{lemma}\cite[Lemma 3.6]{Gr}\label{lemma_function_phi}
Let $M$ be a set and let (P) be a certain property of subsets of $M$. Suppose that $\phi:\ex(M)\to \ex(M)$ is a function satisfying the following two conditions:
\begin{enumerate}[(a)]
\item $|\phi(A)|\leq \max\{\omega,|A|\}$ for all $A\in \ex(M)$ and
\item  If for some ordinal $\kappa$, $A$ is an increasing union of sets $\{A_\alpha:\alpha<\kappa\}$, then $\phi(A)$ is the increasing union of the sets $\{\phi(A_\alpha):\alpha<\kappa\}$.
\end{enumerate}
Then $M$ has the property (P) whenever (P) satisfies the following two conditions:
\begin{enumerate}[(i)]
\item Every countable $\phi$-closed subset of $M$ has (P) and
\item Every increasing union of $\phi$-closed sets satisfying (P) satisfies (P) as well.
\end{enumerate}
\end{lemma}

The next lemma plays a key role in proving Theorem \ref{main}. The proof is a slight modification of a reasoning of Gruenhage (see \cite[Theorem 3.7]{Gr}). For the reader’s convenience, we give a complete argument. 

\begin{lemma}\label{main lemma}
Let $K$ be a zero-dimensional compact space and let $X\subseteq K$ be a dense scattered subspace of $K$. Let $p\in X$ be a W-point in $X$. Assume further that for any clopen subset $A$ of $K$ with $p\notin A$, the intersection $A\cap X$ is $\sigma$-compact. If $\mathscr{U}$ is a family of clopen subsets of $K$ such that $\bigcup\mathscr{U}=K\setminus\{p\}$, then $\UU$ has an open refinement $\mathscr{V}$ such that $\bigcup\VV=\bigcup\UU$ and for any $x\in X$ the collection $\{V\in \VV:x\in V\}$ is finite. 
\end{lemma}
\begin{proof}
Let $\sigma$ be a winning strategy for player I in the $W$-game played at $p$. % comment on convension $\sigma(F)$ and player I pick open subsets of K
For any $\AAA\subseteq \UU$ we set
$$\AAA^*=\left\{A\setminus\bigcup\FF:A\in \AAA \mbox{ and }\FF\in[\AAA]^{<\omega}\right\}.$$
Note that for any $\AAA\subseteq \UU$, the family $\AAA^*$ consists of clopen subsets of $K$ and  $\AAA\subseteq \AAA^*\subseteq \UU^*$. It is also clear that $|\AAA^*|\leq \max\{\omega, |\AAA|\}$. If $A\in \UU^*$, then $A$ is a clopen subset of $K$ missing $p$. So by our assumption, the set $A\cap X$ is $\sigma$-compact (empty if $A=\emptyset$), for any $A\in \UU^*$. Thus, for each $A\in \UU^*$ we can fix a countable family $\{K_n(A):n\in\omega\}$ of compact subsets of $X$ (possibly empty) satisfying
$$A\cap X=\bigcup_{n\in \omega}K_n(A).$$
In particular each of the sets $K_n(A)$ is scattered being a subset of a scattered space $X$. For $\AAA\subseteq \UU$, let
\begin{align*}
&\KK(\AAA^*)=\{K_n(A):A\in \AAA^* \mbox{ and } n\in \omega\} \mbox{ and }\\
&\widehat{\KK(\AAA^*)}=\left\{\bigcap \FF:\FF \in\left[\KK(\AAA^*)\right]^{<\omega} \right\}.
\end{align*}
In other words, $\widehat{\KK(\AAA^*)}$ is the family of all finite intersections of sets of the form $K_n(A)$, where $A\in \AAA^*$ and $n\in \omega$.
For a compact scattered set $C$ we define
\begin{equation*}
Z(C)=
  \left\{\begin{aligned}
  &\emptyset &&\mbox{if }C=\emptyset\\
&C^{(\gamma)}  &&\mbox{if }C\neq \emptyset\mbox{ and } ht(C)=\gamma+1.
\end{aligned}
 \right.
\end{equation*}
Note that $Z(C)$ is a finite set.
For $\AAA\subseteq \UU$, we set
$$Z(\mathscr{A})=\bigcup\left\{Z(C):C \in \widehat{\KK(\AAA^*)}\right\}.$$

If $\AAA\subseteq \UU$ and $s\in Z(\AAA)^{<\omega}$ is a $\sigma$-sequence, then the compact set $K\setminus \sigma(s)$ is contained in $K\setminus\{p\}=\bigcup\UU$ (note that $\sigma(s)$ is well defined because $s$ is a $\sigma$-sequence in $Z(\AAA)$, so in particular it is a $\sigma$-sequence in $X$). Hence, we can find a finite subfamily $\EE(s)\subseteq \mathscr{U}$ such that $K\setminus\sigma(s)\subseteq \bigcup \EE(s)$.  

Consider the following property (P) of a family $\AAA\subseteq \UU$:
\begin{equation}\tag{P}\label{propertyP}
\AAA \mbox{ has a point-finite in }X \mbox{ open refinement } \WW \mbox{ satisfying }\bigcup\WW=\bigcup\AAA.
\end{equation}

In order to prove the lemma, we need to show that the family $\UU$ satisfies the property \eqref{propertyP}. To this end we will use Lemma \ref{lemma_function_phi} with $M=\UU$.
Define a map $\phi:\ex(\mathscr{U})\to \ex(\mathscr{U})$ by the formula
$$\phi(\AAA)=\bigcup\{\EE(s):s\in Z(\AAA)^{<\omega} \mbox{ is a $\sigma$-sequence}\}.$$
It is readily seen that
\begin{equation*}\label{a}
    |\phi(\AAA)|\leq \max\{\omega,|\AAA|\}, \mbox{ for every } \AAA\subseteq \UU.   
\end{equation*}
It is also straightforward to verify that if, for some ordinal $\kappa$, $\AAA$ is an increasing union of families $\{\AAA_\alpha:\alpha<\kappa\}$, then $\phi(\AAA)$ is an increasing union of families $\{\phi(\AAA_\alpha):\alpha<\kappa\}$. So the map $\phi$ satisfies the conditions (a) and (b) of Lemma \ref{lemma_function_phi}. It remains to check that the property \eqref{propertyP} satisfies the conditions (i) and (ii) in Lemma \ref{lemma_function_phi}.

For (i), let $\mathscr{B}\subseteq \UU$ be a countable subfamily of $\UU$ (not necessarily $\phi$-closed). Enumerate $\mathscr{B}=\{B_1,B_2,\ldots\}$. Put $V_1=B_1$ and $V_n=B_n\setminus\bigcup_{i=1}^{n-1}B_i$ for $n>1$. Then $\mathscr{V}=\{V_1,V_2,\ldots\}$ is a pairwise disjoint clopen refinement of $\mathscr{B}$ such that $\bigcup\VV=\bigcup\mathscr{B}$. In particular, $\mathscr{B}$ has the property \eqref{propertyP}.

Let us verify condition (ii). To this end, fix an ordinal $\kappa$ and an increasing collection of $\phi$-closed families $\{\BB_\alpha\subseteq \UU:\alpha<\kappa\}$, each of which satisfying the property \eqref{propertyP}. Let $\BB=\bigcup_{{\alpha<\kappa}}\BB_\alpha$. We need to show that $\BB$ satisfies \eqref{propertyP} as well.
The proof of the following claim is identical as the proof of Claim 1 in \cite[Theorem 3.7]{Gr}. We present the argument here for the sake of completeness.
\begin{claim}
If $\AAA\subseteq \UU$ is a $\phi$-closed family, then $\overline{Z(\AAA)}\subseteq \bigcup\AAA\cup\{p\}$ (the closure is taken in $K$).
\end{claim}
\begin{proof}
Striving for a contradiction, suppose that there is a point $x\neq p$ such that
\begin{equation}\tag{$\ast$}\label{*}
x\in \overline{Z(\AAA)}\setminus \bigcup\AAA.
\end{equation}

Find a clopen neighborhood $W$ of $x$ in $K$ such that $p\notin W$. Inductively, construct points $x_1,x_2,\ldots\in X$ and open in $K$ sets $U_1,U_2,\ldots$ such that:
\begin{enumerate}[(1)]
\item $x_n\in Z(\AAA)\cap W\cap U_n$ for all $n$,
    \item $U_1=\sigma(\emptyset)$ and $U_{n+1}=\sigma(x_1,\ldots ,x_n)$ for $n\geq 1$.
    \end{enumerate}

We begin by letting $U_1=\sigma(\emptyset)$. 
Since $\emptyset\in Z(\AAA)^{<\omega}$ is a $\sigma$-sequence and $\AAA$ is $\phi$-closed, we get $\EE(\emptyset)\subseteq \phi(\AAA)\subseteq \AAA$. Additionally, by definition of the family $\EE(\emptyset)$, we have $K\setminus U_1\subseteq \bigcup\EE(\emptyset)$. Hence, $K\setminus U_1\subseteq \bigcup\AAA$. We now infer from \eqref{*} that $x\in U_1$. So the set $W\cap U_1$ is an open neighborhood of $x$ in $K$ and thus, by \eqref{*}, we can find $x_1\in Z(\AAA)\cap W\cap U_1$. Let $k\geq 1$ and suppose that we have constructed points $\{x_1,\ldots ,x_k\}$ and open sets $U_1,\ldots ,U_k$ in such a way that the conditions (1) and (2) hold for $n=1,\ldots ,k$. In particular, $(x_1,\ldots ,x_k)$ is a $\sigma$-sequence. Let $U_{k+1}=\sigma(x_1,\ldots ,x_k)$. By (1), $(x_1,\ldots , x_k)\in Z(\AAA)^{<\omega}$. Since the family $\AAA$ is $\phi$-closed, we have $\EE((x_1,\ldots , x_k))\subseteq \phi(\AAA)\subseteq \AAA$.
Additionally, by definition of the family $\EE((x_1,\ldots , x_k))$, we have $K\setminus U_{k+1}\subseteq \bigcup\EE((x_1,\ldots , x_k))$. Hence, $K\setminus U_{k+1}\subseteq \bigcup\AAA$.
Condition \eqref{*} implies that $x\in U_{k+1}$ and we can find $x_{k+1}\in Z(\AAA)\cap W\cap U_{k+1}$. This finishes the inductive construction.

Conditions (1) and (2) imply that $(x_1,x_2,\ldots)$ is an infinite $\sigma$-sequence. Hence it must be convergent to $p$, because $\sigma$ is a winning strategy for player I. This however is impossible because, according to (1), the set $\{x_1,x_2,\ldots\}$ is contained in the set $W$ which is closed and misses $p$.
\end{proof}

For any $\alpha<\kappa$ define
$$\BB_\alpha'=\BB_\alpha\setminus\bigcup_{\beta<\alpha}\BB_\beta.$$
The family $\bigcup_{\beta<\alpha}\BB_\beta$ is $\phi$-closed being an increasing union of $\phi$-closed families (cf. Lemma \ref{lemma_function_phi} (b)). So, according to Claim 1, for each $\alpha<\kappa$ we have
$$\overline{Z\Big( \bigcup_{\beta<\alpha}\BB_\beta \Big)}\subseteq \bigcup_{\beta<\alpha}\BB_\beta\cup\{p\}.$$
Hence, if $\alpha<\kappa$ and  $A\in \BB_\alpha'$, then the compact set $A\cap \overline{Z\Big( \bigcup_{\beta<\alpha}\BB_\beta \Big)}\subseteq \bigcup_{\beta<\alpha}\BB_\beta$, because $p\notin \bigcup \UU\supseteq A$. Since the set $A\cap \overline{Z\Big( \bigcup_{\beta<\alpha}\BB_\beta \Big)}$ is compact, we can find a finite subfamily $\FF_\alpha(A)$ of $\bigcup_{\beta<\alpha}\BB_\beta$ with $$A\cap \overline{Z\Big( \bigcup_{\beta<\alpha}\BB_\beta \Big)}\subseteq \bigcup \FF_\alpha(A).$$
For a given $\alpha<\kappa$ we put
$$\BB_\alpha''=\left\{A\setminus\bigcup\FF_\alpha(A):A\in\BB_\alpha' \right\}.$$
Clearly, $\BB_\alpha''\subseteq \BB_\alpha^*$ and
\begin{equation}\tag{$\ast\ast$}\label{**}
 B\cap \overline{Z\Big( \bigcup_{\beta<\alpha}\BB_\beta \Big)}=\emptyset, \mbox{ for every }B\in \BB_\alpha''
\end{equation}
We have the following:
\begin{claim}
The family $\{\bigcup\BB_\alpha'':\alpha<\kappa\}$ is point finite in $X$.
\end{claim}
\begin{proof}
Striving for a contradiction, suppose that there is $x\in X$ and an increasing sequence of ordinals $\alpha_1<\alpha_2<\ldots$ below $\kappa$ such that $x\in \bigcup\BB_{\alpha_i}''$ for $i=1,2,\ldots$.
For each $i$ find $B_i\in \BB_{\alpha_i}''$ such that $x\in B_i$. Since, for all $i$, $B_i\in B_{\alpha_i}''\subseteq \BB_{\alpha_i}^*$, there is
a compact set $K_{n_i}(B_i)\subseteq B_i\cap X$ such that $x\in K_{n_i}(B_i)$. For $k\geq 1$, let
$$C_k= K_{n_1}(B_1)\cap\cdots\cap K_{n_k}(B_k).$$
The sets $C_k$ are nonempty (because $x\in C_k$ for all $k$), compact and scattered being subsets of $X$. So for each $k$, there is an ordinal $\eta_k$ such that $ht(C_k)=\eta_k+1$. The sequence $C_1\supseteq C_2\supseteq\ldots$ is decreasing, so $\eta_{k+1}\leq \eta_k$. Hence, there must be $m$ such that $\eta_{k}=\eta_m$, for $k\geq m$. It follows that
$Z(C_m)\supseteq Z(C_{m+1})$ whence
$$C_{m+1}\cap Z(C_m)\supseteq Z(C_{m+1}).$$
Since the latter set is nonempty, we have
\begin{equation}\tag{$\dagger$}\label{dagger}
\emptyset\neq C_{m+1}\cap Z(C_m)\subseteq K_{n_{m+1}}(B_{m+1})\cap Z(C_m)\subseteq B_{m+1}\cap Z(C_m).
\end{equation}

On the other hand, if $i\leq m$, then 
$B_i\in \BB_{\alpha_i}''\subseteq \BB_{\alpha_i}^*\subseteq \BB_{\alpha_{m}}^*$. So, for $i\leq m$, we have $K_{n_i}(B_{i})\in \KK(\BB_{\alpha_m}^*)$ whence $C_{m}=K_{n_1}(B_1)\cap\cdots\cap K_{n_{m}}(B_{m})\in \widehat{\KK(\BB_{\alpha_m}^*)}$. It follows that $$Z(C_m)\subseteq Z(\BB_{\alpha_m})\subseteq Z\Big(\bigcup_{\beta<\alpha_{(m+1)}}\BB_{\beta}\Big).$$
Combining this with $B_{m+1}\in \BB_{\alpha_{(m+1)}}''$, we infer from \eqref{**} that $B_{m+1}\cap Z(C_m)=\emptyset$, contradicting \eqref{dagger}.
\end{proof}
Returning to the proof of the lemma, recall that, by our assumption the family $\BB$ is an increasing union $\BB=\bigcup_{\alpha<\kappa}\BB_\alpha$ and the familes $\BB_\alpha$ satisfy \eqref{propertyP}. Thus, for each $\alpha<\kappa$, we can find an open refinement $\WW_\alpha$ of $\BB_\alpha$ such that $\bigcup\WW_\alpha=\bigcup\BB_\alpha$ and $\WW_\alpha$ is point-finite in $X$. Let
$$\VV_\alpha=\left\{ W\cap \left( \bigcup\BB_\alpha''\right):W\in \WW_\alpha \right\},$$
for every $\alpha<\kappa$. We claim that the family $\WW=\bigcup_{\alpha<\kappa}\VV_\alpha$ is an open refinement of $\BB$ such that $\bigcup\WW=\bigcup\BB$ and, for each $x\in X$, the set $\{W\in \WW:x\in W\}$ is finite. 

First, note that since the family $\BB_\alpha''$ consists of clopen subsets of $K$ and $\WW_\alpha$ is a family of open subsets of $K$, the family $\WW$ consists of open sets. Using the fact that $\WW_\alpha$ is a refinement of $\BB_\alpha$, it is immediate to see that $\WW$ refines $\BB$. In particular, $\bigcup\WW\subseteq \bigcup\BB$. To check that $\bigcup\WW=\bigcup\BB$, fix $x\in \bigcup\BB$ and consider $\lambda=\min\{\beta<\kappa:x\in \BB_\beta\}$. We can find $A\in \BB_\lambda'$ with $x\in A$. By definition of $\lambda$, we have $x\in A\setminus \bigcup\FF_\lambda(A)$ (recall that $\FF_\lambda(A)\subseteq \bigcup_{\beta<\lambda}\BB_\beta$), while $A\setminus \bigcup\FF_\lambda(A)\in \BB_\lambda''$. Hence $x\in \bigcup \BB_\lambda''$. Moreover, $\bigcup\WW_\lambda=\bigcup\BB_\lambda$, so there is $W\in \WW_\lambda$ with $x\in W$ whence $x\in W\cap \left( \bigcup\BB_\lambda''\right)\subseteq \bigcup\VV_\lambda\subseteq \bigcup \WW$. Applying Claim 2 and the fact that each family $\WW_\alpha$ is point-finite in $X$, one easily verifies that $\WW$ is point-finite in $X$. We conclude that the family $\BB$ has \eqref{propertyP}, and thus, \eqref{propertyP} satisfies condition (ii) in Lemma \ref{lemma_function_phi}. This finishes the proof.
\end{proof}

The following proposition is known. It is a direct consequence of \cite[Theorem 6]{Tel} and \cite[Theorem 5.1.2]{En}.

\begin{prop}\label{fact}
If $X$ is a Lindel\"of scattered space, then $X$ is zero-dimensional. In particular, $X$ has a zero-dimensional compactification.
\end{prop}

Now, with Lemmas \ref{lemma Tkachuk} and \ref{main lemma} in hand, we can prove our main result.

\begin{proof}[Proof of Theorem \ref{main}]
 The proof goes by induction on $\alpha=ht(X)$, where $ht(X)$ is the Cantor-Bendixson rank of $X$. If $\alpha=1$ then $X$ is discrete and hence countable being Lindel\"of. So $X$ is $\sigma$-compact in this case. Suppose that any Lindel\"of scattered $W$-space of rank less than $\alpha$ is $\sigma$-compact and let $ht(X)=\alpha$. If $\alpha$ is a limit ordinal then we are easily done because in this case every point $x\in X$ has a clopen neighborhood (cf. Proposition \ref{fact}) $U_x$ such that $ht(U_x)<\alpha$. By the inductive assumption each $U_x$ is $\sigma$-compact. The Lindel\"of property of $X$ implies that $X$ can be covered by countably many sets $U_x$, whence $X$ is $\sigma$-compact.

 Suppose that $\alpha=\beta+1$ for some ordinal $\beta$. The top level $X^{(\beta)}$ of $X$ is closed and discrete. We can assume, without loss of generality that $X^{(\beta)}$ is a singleton. Indeed, in the general case (i.e, when $X^{(\beta)}$ is not necessarily a one-element set), for each $x\in X$ we can find a clopen neighborhood $U_x$ (see Proposition \ref{fact}) such that $U_x\cap X^{(\beta)}=\emptyset$ if $x\notin X^{(\beta)}$ (because $X^{(\beta)}$ is closed in $X$) and $U_x\cap X^{(\beta)}$ is a singleton for $x\in X^{(\beta)}$ (by discreteness of $X^{(\beta)}$). By the Lindel\"of property, $X$ a union of countably many sets $U_x$. Now, if $x\notin X^{(\beta)}$, then $U_x$ is $\sigma$-compact by the inductive assumption and hence, if we knew that scattered spaces, having the one-element top level, are $\sigma$-compact, we would be done because $U_x$ for $x\in X^{(\beta)}$ are such spaces.
So assume that $X^{(\beta)}=\{p\}$ for some $p\in X$. According to \ref{fact}, the space $X$ has a zero-dimensional compactification $K$. We can assume that $X\subseteq K$.
Since $X$ is a $W$-space, the point $p$ is a $W$-point in $X$. If $A\subseteq K$ is a clopen subset of $K$ such that $p\notin A$, then $A\cap X$ is Lindel\"of (being a closed subspace of a Lindel\"of space $X$), scattered of rank less than $\alpha$, because $p\notin A$. So by the inductive assumption, $A\cap X$ is $\sigma$-compact. This means that $X\subseteq K$ and $p$ satisfy the assumptions of Lemma \ref{main lemma}. For each $x\in K\setminus\{p\}$ find a clopen neighborhood $U_x$ of $x$ with $p\notin U_x$ and let 
$$\UU=\{U_x:x\in K\setminus \{p\}\}.$$
From Lemma \ref{main lemma}, we infer that $\UU$ has an open refinement $\VV$ such that $\bigcup \VV=\bigcup\UU=K\setminus \{p\}$ and the family $\VV$ is point-finite in $X$. Enumerate faithfully $\VV=\{V_\gamma:\gamma<\kappa\}$. Fix an arbitrary $\gamma<\kappa$. Since $\VV$ is a refinement of $\UU$, the set $V_\gamma$ is contained in some clopen set $U\in \UU$. Hence $p\notin U\supseteq \overline{V_\gamma}$. This and the inductive assumption imply that $\overline{V_\gamma}\cap X$ is $\sigma$-compact for every $\gamma<\kappa$. Hence, the set
$$G_\gamma=V_\gamma\cap \left(K\setminus \left(\overline{V_\gamma}\cap X\right)\right)=V_\gamma\setminus X$$
is a $G_\delta$-subset of $K$ for all $\gamma<\kappa$.

Now, since $\VV=\{V_\gamma:\gamma<\kappa\}$ is point-finite in $X$ and $G_\gamma\subseteq V_\gamma$ for every $\gamma<\kappa$, we can use Lemma \ref{lemma Tkachuk} to find a $G_\delta$-subset $G$ of $K$ such that
$$\bigcup_{\gamma<\kappa}G_\gamma\subseteq G\subseteq  K\setminus X.$$
But $\bigcup_{\gamma<\kappa}G_\gamma=\left(\bigcup\VV\right)\setminus X=\left(K\setminus \{p\}\right)\setminus X=K\setminus X$, so $G=K\setminus X$. We conclude that $X$ is $\sigma$-compact having a $G_\delta$ complement in a compact space.
\end{proof}

\bibliographystyle{siam}
\bibliography{bib.bib}

\end{document}